\documentclass[11pt]{article}
\usepackage{graphicx}
\usepackage{amscd, amssymb}
\usepackage{enumerate}
\usepackage{hyperref}
\usepackage{lscape}
\newenvironment{eq}{\begin{equation}}{\end{equation}}
\newenvironment{proof}{{\bf Proof}:}{\vskip 5mm }

\newtheorem{proposition}{Proposition}[subsection]
\newtheorem{lemma}[proposition]{Lemma}
\newtheorem{definition}[proposition]{Definition}
\newtheorem{theorem}[proposition]{Theorem}

\newtheorem{example}[proposition]{Example}
\newtheorem{remark}[proposition]{Remark}

\newtheorem{problem}[proposition]{Problem}
\newtheorem{construction}[proposition]{Construction}
\newcommand{\llabel}[1]{\label{#1}}
\newcommand{\comment}[1]{}
\newcommand{\sr}{\rightarrow}

\newcommand{\nn}{{\bf N\rm}}

\newcommand{\wt}{\widetilde}

\newcommand{\rtr}{\triangleright}

{\vskip
3mm}

\newcommand{\spc}{{\,\,\,\,\,\,\,}}

\newcommand{\B}{{\bf B}}
\newcommand{\FF}{{\bf F}}
\newcommand{\TT}{{\bf T}}
\renewcommand{\SS}{{\bf S}}
\newcommand{\BD}{{\bf BD}}

\newcommand{\JJ}{{\mathcal J}}

\begin{document}
\parskip = 2mm
\begin{center}
{\bf\Large B-systems}

\vspace{3mm}

{\large\bf Vladimir Voevodsky}\footnote{School of Mathematics, Institute for Advanced Study,
Princeton NJ, USA. e-mail: vladimir@ias.edu}$^,$\footnote{Work on this paper was supported by NSF grant 1100938.}
\vspace {3mm}

{Started September 08, 2009, cont. July 23, 2012, November 20, 2012, December 1, 2012, June 12, 2014, October 10, 2014}  
\end{center}

\begin{abstract}
B-systems are algebras (models) of an essentially algebraic theory that is expected to be constructively equivalent to the essentially algebraic theory of C-systems which is, in turn, constructively equivalent to the theory of contextual categories.  The theory of B-systems is closer in its form to the structures directly modeled by contexts and typing judgements of (dependent) type theories and further away from categories than contextual categories and C-systems. 
\end{abstract}

%

\subsection{Introduction}
In \cite[Def. 2.2]{Csubsystems} we introduced the concept of a C-system. The type of the C-systems is constructively equivalent to the type of contextual categories defined by Cartmell in \cite{Cartmell1} and \cite{Cartmell0} but the definition of a C-system is slightly different from the Cartmell's foundational definition.

The concept of a B-system is introduced in this  paper. It provides an abstract formulation of a structure formed by contexts and ``typing judgements'' of a type theory relative to the operations of context extensions, weakening and substitutions. 

We define B-systems in several steps. First we describe pre-B-systems that are models of an essentially algebraic theory with countable families of sorts operations but no relations. 

Already at this stage we tart to distinguish between unital and non-unital (pre-)B-systems. This distinction continues throughout the paper. While non-unital B-systems have no direct connection to C-systems and therefore no direct connection to categories they have a definition with interesting symmetries and we believe that they are quite interesting in there own right. 

Following the ideas of \cite{Csubsystems}, how to construct a unital pre-B-system from a C-system. This construction is functorial with respect to homomorphisms of C-systems and unital pre-B-systems and moreover defines a full embedding of the category of C-systems to the category of unital pre-B-systems. 

It is more or less clear from the proof of the full embedding theorem that the image of this full embedding consists of unital pre-B-systems whose operations satisfy some algebraic conditions. We suggest a form of these conditions in our definition of a non-unital and then unital B-system (Definitions \ref{2014.10.10.def2a} and \ref{2014.10.10.def2b}). 

We conclude the first part of the paper with a problem (essentially a conjecture) that the image of the full embedding from C-systems to unital pre-B-systems is precisely the class of unital B-systems. A constructive solution to this problem would also provide an explicit construction of a C-system from a unital B-system.

In the second part we describe an approach to the definition of non-unital B-systems that can be conveniently formalized in Coq and that provide a possible step towards the definition of higher B-systems that is B-systems whose component types are of higher h-levels. 

The work on this paper, especially in the part where the axioms of $TT$, $SS$, $TS$ and $ST$ of B-systems  are introduced was influenced and facilitated by recent discussions with Richard Garner and Egbert Rijke. Many other ideas of this work go back to \cite{NTS}.  

The subject of this paper is closely related to the subject of recent notes by John Cartmell \cite{Cartmell2014}. The most important difference between our exposition and that of Cartmell is that we are using the formalism of {\em essentially algebraic} theories while Cartmell uses the formalism of {\em generalized algebraic} theories. While there are important connections between these two kinds of theories there are also important distinctions which we intend to discuss in a future paper. 

I am grateful to The Centre for Quantum Mathematics and Computation (QMAC) and the Mathematical Institute of the University of Oxford for their hospitality during my work on this paper.

\subsection{pre-B-systems}
\begin{definition}
\llabel{2014.10.10.def1}
A non-unital pre-B-system a collection of data of the following form:
\begin{enumerate}
\item for all $n\in\nn$ two set $B_n$ and $\wt{B}_{n+1}$,
\item for all $n\in\nn$ maps of the form:
\begin{enumerate}
\item $ft:B_{n+1}\sr B_n$,
\item $\partial:\wt{B}_{n+1}\sr B_{n+1}$
\end{enumerate}
\item an element $pt\in B_0$,
\item for all $m,n\in \nn$ such that $m\ge n$ maps of the form:
\begin{enumerate}
\item $T:(Y\in B_{{n}+1}, X\in B_{{m}+1}, ft(Y)=ft^{{m}+1-{n}}(X))\sr B_{{m}+2}$,  
\item $\wt{T}:(Y\in B_{{n}+1}, r\in\wt{B}_{{m}+1}, ft(Y)=ft^{{m}+1-{n}}\partial(r))\sr \wt{B}_{{m}+2}$, 
\item $S:(s\in\wt{B}_{{n}+1}, X\in B_{{m}+2}, \partial(s)=ft^{{m}+1-{n}}(X))\sr B_{{m}+1}$, 
\item $\wt{S}: (s\in \wt{B}_{{n}+1}, r\in\wt{B}_{{m}+2}, \partial(s)=ft^{{m}+1-{n}}\partial(r))  \sr \wt{B}_{{m}+1}$,  
\end{enumerate}
\end{enumerate}
\end{definition}
\begin{definition}
\llabel{2014.10.20.def1}
A unital pre-B-system is a non-unital pre-B-system together with, for every $n\ge 0$ of an operation
$$\delta:B_{n+1}\sr \wt{B}_{n+2}$$
\end{definition}
Homomorphisms of non-unital and unital pre-B-systems are defined in the obvious way giving us the corresponding  categories.  Also in the obvious way one defines the concepts of sub-pre-B-systems.

Let $CC$ be a C-system as defined in \cite[Def. 2.2]{Csubsystems}. Recall the following notations. For $X$ such that $l(X)\ge i$ and $f:Y\sr ft^i(X)$ denote by by $f^*(X,i)$  the objects and by $q(f,X,i):f^*(X,i)\sr X$ the morphisms defined inductively by the rule 
$$f^*(X,0)=Y\,\,\,\,\,\,\,\,\,q(f,X,0)=f,$$
$$f^*(X,i+1)=q(f,ft(X),i)^*(X)\,\,\,\,\,\,\,\,\,q(f,X,i+1)=q(q(f,ft(X),i), X).$$
If $l(X)<i$, then $q(f,X,i)$ is undefined since $q(-,X)$ is undefined for $X=pt$ and again, as in the case of $p_{X,i}$,  {\em all of the considerations involving $q(f,X,i)$ are modulo the qualification that $l(X)\ge i$}. 

For $i\ge 1$, $(s:ft(X)\sr X)\in \wt{Ob}$ such that $l(X)\ge i$, and $f:Y\sr ft^i(X)$ let 
$$f^*(s,i):f^*(ft(X),i-1)\sr f^(ft(X),i)$$
be the pull-back of the section $ft(X)\sr X$ along the morphism $q(f,ft(X),i-1)$. We again use the agreement that always when $f^*(s,i)$ is used the condition $l(X)\ge i$ is part of the assumptions. 

One constructs a unital pre-B-system from $CC$ as follows. The B-sets of $CC$ are:
$$B_n(CC)=Ob_n(CC)=\{X\in Ob(CC)\,|\,l(X)=n\}$$
$$\wt{B}_{n+1}(CC)=\wt{Ob}_n(CC)=\{(X,s)\in \wt{Ob}(CC)\,|\,l(X)=n+1\}$$
The definition of $pt$, $ft$ and $\partial$ is obvious. The operations $T$, $\wt{T}$, $S$, $\wt{S}$ and $\delta$ on the B-sets of a C-system are as follows:
\begin{enumerate}
\item $T$ sends $(Y,X)$ such that $ft(Y)=ft^{{m}+1-{n}}(X)$ to  $p_Y^*(X,{m}+1-{n})$,
\item $\wt{T}$ sends $(Y,r)$ such that $ft(Y)=ft^{{m}+1-{n}}\partial(r)$ to $p_Y^*(r,{m}+1-{n})$,
\item $S$ sends $(s,X)$ such that $\partial(s)=ft^{{m}+1-{n}}(X)$ to $s^*(X,{m}+1-{n})$,
\item $\wt{S}$ sends $(s,r)$ such that $\partial(s)=ft^{{m}+1-{n}}\partial(r)$ to $s^*(r,{m}+1-{n})$.
\item $\delta$ sends $X$ to the diagonal section of the projection $p_X^*X\sr X$.
\end{enumerate}
When we need to distinguish between the unital pre-B-system defined by $CC$ and its non-unital analog we will write $uB(CC)$ for the unital version and $nuB(CC)$ for the non-unital one. 

One of the main results of \cite{Csubsystems}, Proposition 4.3 can be reformulated as follows:
\begin{theorem}
\llabel{2014.06.26.th1}
There is a natural bijection between C-subsystems of a C-system CC and unital sub-pre-B-systems of $uB(CC)$. 
\end{theorem}

Another way to define a pre-B-system is from a pair $(R,LM)$ where $R$ is a monad on sets and $LM$ a left module over $R$ with values in sets as in \cite{Cofamodule} . For the pre-B-system $B(R,LM)$ we have
$$B_n(R,LM)=LM(\emptyset)\times\dots\times LM(\{1,\dots,n-1\})$$
$$\wt{B}_{n+1}(R,LM)=B_{n+1}(R,LM)\times R(\{1,\dots,n\})$$
The operations $ft$ and $\partial$ are obvious. The element $pt$ is the only point of the product of the empty family of sets. The rest of the operations are defined as follows. For $E\in LM(\{1,\dots,m\})$ or $E\in R(\{1,\dots,m\})$ and $n\ge 1$ we set:
$$t_n(E)=E[n+1/n,n+2/n+1,\dots,m+1/m]$$
$$s_n(E)=E[n/n+1,n+1/n+2,\dots,m-1/m]$$
\begin{enumerate}
\item Operations $T$:
$$T((E_1,\dots,E_n,F),(E_1,\dots,E_n,E_{n+1},\dots,E_{m+1}))=$$
$$(E_1,\dots,E_n,F,t_{n+1}E_{n+1},\dots,t_{n+1}E_{m+1})$$
\item Operations $\wt{T}$:
$$\wt{T}((E_1,\dots,E_n,F),(E_1,\dots,E_n,E_{n+1},\dots,E_{m+1}, r))=$$
$$(E_1,\dots,E_n,F,t_{n+1}E_{n+1},\dots,t_{n+1}E_{m+1},t_{n+1}r)$$
\item Operations $S$:
$$S((E_1,\dots,E_n,F,s),(E_1,\dots,E_n,F,E_{n+1},\dots,E_{m+1}))=$$
$$(E_1,\dots,E_n,s_n(E_{n+1}[s/n]),\dots,s_n(E_{m+1}[s/n]))$$
\item Operation $\wt{S}$:
$$S((E_1,\dots,E_n,F,s),(E_1,\dots,E_n,F,E_{n+1},\dots,E_{m+1},r))=$$
$$(E_1,\dots,E_n,s_n(E_{n+1}[s/n]),\dots,s_n(E_{m+1}[s/n]),s_n(r[s/n]))$$
\item Operations $\delta$:
$$\delta(E_1,\dots,E_n,E_{n+1})=(E_1,\dots,E_n,E_{n+1},\eta_R(n+1))$$
where $\eta_R$ is the unit of the monad $R$.
\end{enumerate}
Note that the unit of $R$ also participates in the definition of operations $S$ and $\wt{S}$ since the explicit form of the substitution $E\mapsto E[s/n]$ involves $\eta_R$. 

We can form non-unital pre-B-systems using this construction by considering non-unital sub-pre-B-systems in $uB(R,LM)$ (cf. Example \ref{2014.10.20.ex1} below). 

For this pre-B-system as well as for its subsystems and regular quotients we can use notations such as $\Gamma\vdash o:T$ directly since in this case $\Gamma\in B_n$, $T\in LM(\{1,\dots,n\})$ and $o\in R(\{1,\dots,n\})$ are elements of types or sets that do not depend on elements of other types or sets and the substitution is defined on the level of these sets.

If $CC(R,LM)$ is the C-system corresponding to $(R,LM)$ then there is a constructive isomorphism
$$B(CC(R,LM))\cong B(R,LM)$$

The construction $CC\mapsto B(CC)$ is clearly compatible with homomorphisms and defines a functor from the category of C-systems to the category of unital pre-B-systems.
\begin{theorem}
\llabel{2014.10.10.th1}
The functor $CC\mapsto uB(CC)$ is a full embedding.
\end{theorem}
The proof follows from the lemmas below that show that a C-system can be reconstructed from the associated unital pre-B-system. 

We start by introducing intermediate concepts of a B0-systems.
\begin{definition}
\llabel{2014.10.16.def1.fromold}
\llabel{2014.10.16.def1}
A non-unital pre-B-system is called a non-unital B0-system if the following conditions hold:
\begin{enumerate}
\item for all $X\in B_0$ one has $X=pt$.
\item for $Y\in B_{n+1}$, $X\in B_{m+1}$ such that $ft(Y)=ft^{m+1-n}(X)$ and $m\ge n\ge 0$ one has:
\begin{eq}\llabel{oldeq1}
ft(T(Y,X))=\left\{
\begin{array}{ll}
T(Y,ft(X))&\mbox{\rm if $m>n$}\\
Y&\mbox{\rm if $m=n$}
\end{array}
\right.
\end{eq}
\item for $Y\in B_{n+1}$, $r\in \wt{B}_{m+1}$ such that $ft(Y)=ft^{m+1-n}\partial(r)$ and $m\ge n\ge 0$ one has:
\begin{eq}
\partial(\wt{T}(Y,r))=T(Y,\partial(r))
\end{eq}
\item for $s\in \wt{B}_{n+1}$, $X\in \wt{B}_{m+2}$ such that $\partial(s)=ft^{m+1-n}(X)$ and $m\ge n\ge 0$ one has:
\begin{eq}
ft(S(s,X))=\left\{
\begin{array}{ll}
S(s,ft(X))&\mbox{\rm if $m>n$}\\
ft(Y)&\mbox{\rm if $m=n$}
\end{array}
\right.
\end{eq}
\item for $s\in \wt{B}_{n+1}$, $r\in \wt{B}_{m+2}$ such that $\partial(s)=ft^{m+1-n}\partial(r)$ and $m\ge n\ge 0$ one has:
\begin{eq}
\partial(\wt{S}(s,r))=S(s,\partial(r))
\end{eq}
\item
\end{enumerate}
\end{definition}
\begin{definition}
\llabel{2014.10.20.def2}
A unital pre-B-system is called a unital B0-system if the underlying non-unital pre-B-system is 
a non-unital B0-system and for all $i\ge 0$, $X\in B_{n+1}$ one has
\begin{eq}
\llabel{2009.12.27.eq1}
\partial(\delta(X))=T(X,X)
\end{eq}
\end{definition}
\begin{lemma}
\llabel{l2014.10.10.l1}
Let $B$ be a unital pre-B-system of the form $uB(CC)$. Then $B$ is a unital B0-system.
\end{lemma}
\begin{proof}
Straightforward.
\end{proof}

From now on in this section we assume that we consider a unital B0-system. Let us denote by 
$$T_j:(B_{{n}+j})_{ft^j}\times_{ft^{{m}+1-{n}}} (B_{{m}+1})\sr B_{{m}+1+j}$$
$$\wt{T}_j:(B_{{n}+j})_{ft^j}\times_{ft^{{m}+1-{n}}\partial} (\wt{B}_{{m}+1})\sr \wt{B}_{{m}+1+j}$$
the maps which are defined inductively by 
\begin{eq}
T_j(Y,X)=\left\{
\begin{array}{ll}
X&\mbox{\rm if $j=0$}\\
T(Y,T_{j-1}(ft(Y),X))&\mbox{\rm if $j>0$}\\
\end{array}
\right.
\,\,\,\,\,\,\,\,\,\,\,\,
\wt{T}_j(Y,s)=\left\{
\begin{array}{ll}
s&\mbox{\rm if $j=0$}\\
\wt{T}(Y,\wt{T}_{j-1}(ft(Y),s))&\mbox{\rm if $j>0$}\\
\end{array}
\right.
\end{eq}
Note that for any $i=0,\dots,j$ we have 
$$T_j(Y,X)=T_i(Y,T_{j-i}(ft^i(Y),X))$$
and
$$\wt{T}_j(Y,s)=\wt{T}_i(Y,\wt{T}_{j-i}(ft^i(Y),s))$$
\begin{lemma}
\llabel{Tnft}
One has 
$$T_j(Y,ft(X))=ft(T_j(Y,X))$$
\end{lemma}
\begin{proof}
For $n=0$ the statement is obvious. For $n>0$ we have by induction on $j$
$$T_j(Y,ft(X))=T(Y,T_{j-1}(ft(Y),ft(X)))=T(Y,ft(T_{j-1}(ft(Y),X)))=$$
$$=ft(T(Y,T_{j-1}(ft(Y),X)))=ft(T_j(Y,X)).$$
\end{proof}
Let $f:Y\sr X$ be a morphism such that $Y\in B_n$ and $X\in B_m$. Define a sequence $(s_1(f),\dots,s_m(f))$ of elements of $\wt{B}_{n+1}$ inductively by the rule 
$$(s_1(f),\dots,s_m(f))=(s_1(ft(f)),\dots,s_{m-1}(ft(f)),s_f)=(s_{ft^{m-1}(f)},\dots,s_{ft(f)},s_f)$$
where $ft(f)=p_X f$ and $s_f$ is the $s$-operation of \cite[Def. 2.2]{Csubsystems}.  For $m=0$ we start with the empty sequence. This construction can be illustrated by the following diagram for $f:Y\sr X$ where $X\in B_4$:
\begin{eq}
\begin{CD}
Y @>s_4(f)>>   Z_{4,3} @>>> Z_{4,2} @>>> Z_{4,1} @>>> T_n(Y,X) @>>> X\\
@. @VVV @VVV  @VVV @VVV @VVV\\
{} @. Y @>s_{3}(f)>> Z_{3,2} @>>> Z_{3,1} @>>> T_n(Y,ft(X))@>>> ft(X)\\
@. @. @VVV @VVV @VVV @VVV\\
{} @. {} @.  Y @>s_{2}(f)>> Z_{2,1}@>>> T_n(Y,ft^2(X)) @>>>  ft^2(X)\\
@. @. @.  @VVV @VVV @VVV\\
{} @. {} @. {} @. Y @>s_1(f)>>  T_n(Y,ft^{3}(X)) @>>>  ft^{3}(X)\\
@. @. @.  @. @VVV @VVV\\
{} @. {} @. {} @.  @.  Y @>>>  pt\\
\end{CD}
\end{eq}
which is completely determined by the condition that the squares are the canonical ones and the composition of morphisms in the $i$-th arrow from the top is $ft^i(f)$. For the objects $Z_i^j$ we have:
\begin{eq}
\begin{array}{lll}
Z_{4,1}=S(s_1(f), T_n(Y,X))&Z_{4,2}=S(s_2(f),Z_{4,1})&Z_{4,3}=S(s_3(f),Z_{4,2})\\\\
Z_{3,1}=S(s_1(f),T_n(Y,ft(X)))&Z_{3,2}=S(s_2(f),Z_{3,1})&\\\\
Z_{2,1}=S(s_1(f),T_n(Y,ft^2(X)))&&\\
\end{array}
\end{eq}

A simple inductive argument similar to the one in the proof of \cite[Lemma 4.1]{Csubsystems} show that if $f,f':Y\sr X$ are two morphisms such that $X\in B_m$ and $s_i(f)=s_i(f')$ for $i=1,\dots,m$ then $f=f'$.  Therefore, we may consider the set $Mor(CC)$ of morphisms of $CC$ as a subset in $\amalg_{n,m\ge 0} B_n\times B_m\times \wt{B}_{n+1}^m$.

Let us show how to describe this subset in terms of the operations introduced above. 
\begin{lemma}
\llabel{2009.11.07.l1}
An element $(Y,X,s_1,\dots,s_m)$ of $B_n\times B_m\times \wt{B}_{n+1}^m$ corresponds to a morphism if and only if the element $(Y,ft(X),s_1,\dots,s_{m-1})$ corresponds to a morphism and $\partial(s_m)=Z_{m,m-1}$ where $Z_{m,i}$ is defined inductively by the rule:
$$Z_{m,0}=T_n(Y,X)\,\,\,\,\,\,\,\,\,\,Z_{m,i+1}=S(s_{i+1},Z_{m,i})$$
\end{lemma}
\begin{proof}
Straightforward from the example considered above. 
\end{proof}
Let us show now how to identify the canonical morphisms $p_{X,i}:X\sr ft^i(X)$ and in particular the identity morphisms. 
\begin{lemma}
\llabel{2009.11.10.l1}
Let $X\in B_m$ and $0\le i\le m$. Let $p_{X,i}:X\sr ft^i(X)$ be the canonical morphism. Then one has:
$$s_j(p_{X,i})=\wt{T}_{m-j}(X,\delta_{ft^{m-j}(X)})\,\,\,\,\,\,\,\,\,\,    j=1,\dots, m-i $$
\end{lemma}
\begin{proof}
Let us proceed by induction on $m-i$. For $i=m$ the assertion is trivial. Assume the lemma proved for $i+1$.  Since $ft(p_{X,i})=p_{X,i+1}$ we have $s_j(p_{X,i})=s_j(p_{X,i+1})$ for $j=1,\dots, m-i-1$. It remains to show that
\begin{eq}
\llabel{2009.11.10.eq1}
s_{m-i}(p_{X,i})=\wt{T}_{i}(X,\delta_{ft^{i}(X)})
\end{eq}
By definition $s_{m-i}(p_{X,i})=s_{p_{X,i}}$ and (\ref{2009.11.10.eq1}) follows from the commutative diagram:
$$
\begin{CD}
X @>>> ft^i(X)\\
@Vs_pVV @VV\delta_{ft^i(X)}V\\
p_{X,i+1}^*(ft^i(X)) @>>> p_{ft^i(X)}^*(ft^i(X)) @>>> ft^i(X)\\
@VVV @VVV @VVp_{ft^i(X)}V\\
 X @>>> ft^i(X) @>>> ft^{i+1}(X)
\end{CD}
$$
where $p=p_{X,i}$.
\end{proof}
\begin{lemma}
\llabel{2009.11.10.l2}
Let $(X,s)\in \wt{B}_{m+1}$, $Y\in B_n$ and $f:Y\sr ft(X)$. Define inductively
$(f,i)^*(s)\in \wt{B}_{n+m+1-i}$ by the rule
$$(f,0)^*(s)=\wt{T}_n(Y,s)$$
$$(f,i+1)^*(s)=\wt{S}(s_{i+1}(f), (f,i)^*(s))$$
Then $f^*(s)=(f,m)^*(s)$.
\end{lemma}
\begin{proof}
It follows from the diagram:
$$
\begin{CD}
Y @>s_m(f)>> * @>>> \dots @>>>* @>>> * @>>> ft(X)\\
@Vf^*(s)VV @VV(f,m-1)^*(s)V @. @VV(f,1)^*(s)V @VV(f,0)^*(s)V @VVsV\\
* @>>> * @>>> \dots @>>>* @>>> * @>>> X\\
@VVV @VVV @. @VVV @VVV @VVV\\
Y @>s_m(f)>> * @>>> \dots @>>>* @>>> * @>>> ft(X)\\
@. @VVV @. @VVV @VVV @VVV\\
{} @. Y  @>s_{m-1}(f)>> \dots @>>> * @>>> * @>>> ft^2(X)\\
@. @. @. @VVV @VVV @VVV\\
{} @. {} @. {} @. \dots @. \dots @. \dots \\
@. @. @. @VVV @VVV @VVV\\
 {} @. {} @. {} @. Y @>s_1(f)>> * @>>> ft^{m-1}(X)\\
@. @. @. @. @VVV @VVV\\
 {} @. {} @. {} @. {} @. Y @>>> pt
\end{CD}
$$
\end{proof}
\begin{lemma}
\label{2009.11.10.l3}
Let $g:Z\sr Y$, $f:Y\sr X$ and $X\in B_m$. Then $s_i(fg)=g^*s_i(f)$.
\end{lemma}
\begin{proof}
It follows immediately from the equations $s_{fg}=g^*s_f$ and $ft(fg)=ft(f)g$.
\end{proof}

\begin{lemma}
\llabel{2009.11.10.l4a}
Let $f:Y\sr ft(X)$ be a morphism, $Y\in B_n$ and $X\in B_{m+1}$. Define $(f,i)^*(X)$ inductively by the rule:
$$(f,0)^*(X)=T_n(Y,X)$$
$$(f,i+1)^*(X)=S(s_{i+1}(f),(f,i)^*(X))$$
Then $f^*(X)=(f,m)^*(X)$.
\end{lemma}
\begin{proof}
Similar to the proof of Lemma \ref{2009.11.10.l2}.
\end{proof}
\begin{lemma}
\llabel{2009.11.10.l4b}
Let $f:Y\sr ft(X)$ be a morphism, $Y\in B_n$ and $X\in B_{m+1}$. 
Then 
$$s_i(q(f,X))=\left\{
\begin{array}{ll}
\wt{T}(f^*X,s_i(f))&\mbox{\rm if $i\le m$}\\\\
\wt{T}(f^*X,\delta_X)&\mbox{\rm if $i=m+1$}
\end{array}
\right.
$$
\end{lemma} 
\begin{proof}
We have $s_i(q(f,X))=s_{ft^{m+1-i}(q(f,X))}$. For $i\le m$ we have 
$$ft^{m+1-i}(q(f,X))=ft^{m-i}(f)p_{f^*X}$$
Therefore,
$$s_{ft^{m+1-i}(q(f,X))}=s_{ft^{m-i}(f)p_{f^*X}}=p_{f^*X}^*s_{ft^{m-i}(f)}=\wt{T}(f^*X,s_i(f))$$
and for $i=m+1$ we have 
$$s_i(q(f,X))=s_{q(f,X)}=p_{f^*X}^*(\delta_X)=\wt{T}(f^*X,\delta_X).$$
\end{proof}
The lemmas proved above show that a C-system can be reconstructed from the sets $B_n, \wt{B}_{n+1}$ and operations $ft$, $\partial$, $\delta$, $T$, $\wt{T}$, $S$ and $\wt{S}$. This completes our proof of Theorem \ref{2014.10.10.th1}.

\subsection{B-systems}
The next question that we want to address is the description of the image of the functor $CC\mapsto uB(CC)$. To make this question more precise we introduce below the concepts of non-unital and unital B-systems and formulate a problem whose solution would imply that the functor $CC\mapsto uB(CC)$ defines an equivalence between the category of C-systems and the full subcategory of the category of unital pre-B-systems that consists of unital B-systems. 

For $Y\in B_i$ let $B(Y)_j$ denote the subset of $B_{i+j}$ that consists of $X$ such that $ft^j(X)=Y$. In particular $B(Y)_0$ is the one point subset $\{Y\}$. Let also $\wt{B(Y)}_j$ denote the subset of $\wt{B}_{i+j}$ that consists of $r$ such that $ft^j(\partial(r))=Y$.

Then the operations $T$, $\wt{T}$, $S$ and $\wt{S}$ can be seen as follows:
$$T(Y,-):B(ft(Y))_*\sr B(Y)_*$$
$$\wt{T}(Y,-):\wt{B}(ft(Y))_*\sr \wt{B}(Y)_*$$
$$S(s,-):B(\partial(s))_*\sr B(ft(\partial(s)))_*$$
$$\wt{S}(s,-):\wt{B}(\partial(s))_*\sr \wt{B}(ft(\partial(s)))_*$$
\begin{definition}
\llabel{2014.10.16.def2}
\llabel{was.2014.06.18.eq2.to.eq11}
Let $B$ be a non-unital B0-system. Define the following conditions on $B$:
\begin{enumerate}
\item The TT-condition. For all  $GT\in B_{i+1}$, $GDT'\in B(ft(GT))_{j+1}$ one has
\begin{enumerate}
\item for all $R\in B(ft(GDT'))_*$
$$T(T(GT,GDT'),T(GT,R))=T(GT,T(GDT',R))$$
\item for all $r\in \wt{B}(ft(GDT'))_*$
$$\wt{T}(T(GT,GDT'),\wt{T}(GT,r))=\wt{T}(GT,\wt{T}(GDT',r))$$
\end{enumerate}
\item The SS-condition. For all $s\in \wt{B}_{i+1}$, $s'\in \wt{B(\partial(s))}_{j+1}$ one has
\begin{enumerate}
\item for all $R\in B(\partial(s'))_*$
$$S(\wt{S}(s,s'),S(s,R))=S(s,S(s',R))$$
\item for all $r\in \wt{B}(\partial(s'))_*$
$$\wt{S}(\wt{S}(s,s'),\wt{S}(s,r))=\wt{S}(s,\wt{S}(s',r))$$
\end{enumerate}
\item The TS-condition. For any $GT\in B_{i+1}$ and $s'\in \wt{B(ft(GT))}_{j+1}$ one has
\begin{enumerate}
\item for all $R\in B(\partial(s'))_*$
$$S(\wt{T}(GT,s'),T(GT,R))=T(GT,S(s',R))$$
\item for all $r\in \wt{B}(\partial(s'))_*$
$$\wt{S}(\wt{T}(GT,s'),\wt{T}(GT,r))=\wt{T}(GT,\wt{S}(s',r))$$
\end{enumerate}
\item The ST-condition. For any $s\in \wt{B}_{i+1}$ and $GTDT'\in \wt{B(\partial(s))}_{j+1}$ one has
\begin{enumerate}
\item for all $R\in B(ft(GTDT'))_*$
$$T(S(s,GTDT'),S(s,R))=S(s,T(GTDT',R))$$
\item for all $r\in \wt{B}(ft(GTDT'))_*$
$$\wt{T}(S(s,GTDT'),\wt{S}(s,r))=\wt{S}(s,\wt{T}(GTDT',r))$$
\end{enumerate}
\item The STid-condition. For any $s\in \wt{B}_{i+1}$ one has
\begin{enumerate}
\item for all $R\in B(ft(\partial(s)))_*$
$$S(s,T(\partial(s),R))=R$$
\item for all $r\in \wt{B}(ft(\partial(s)))_*$
$$\wt{S}(s,\wt{T}(\partial(s),r))=r$$
\end{enumerate}
\end{enumerate}
\end{definition}
\begin{definition}
\llabel{2014.10.20.def3}
Let $B$ be a unital B0-system. Define the following conditions on $B$:
\begin{enumerate}
\item The $\delta$T-condition. For any $GT\in B_{i+1}$ and $GDT'\in B(ft(GT))_{j+1}$ one has
$$\wt{T}(GT,\delta(GDT'))=\delta(T(GT,GDT'))$$
\item The $\delta$S-condition. For any $s\in \wt{B}_{i+1}$ and $GTDT'\in B(\partial(s))_{j+1}$ one has
$$\wt{S}(s,\delta(GTDT'))=\delta(S(s,GTDT'))$$
\item The $\delta$Sid-condition. For any $s\in \wt{B}_{i+1}$ one has
$$\wt{S}(s,\delta(\partial(s)))=s$$
\item The S$\delta$T-condition. For any $GT\in B_{i+1}$ one has
\begin{enumerate}
\item for $R\in B(GT)_*$ one has:
$$S(\delta(GT),T(GT,R))=R$$
\item for $r\in \wt{B(GT)}_*$ one has
$$\wt{S}(\delta(GT),\wt{T}(GT,r))=r$$
\end{enumerate}
\end{enumerate}
\end{definition}
\begin{remark}
\llabel{2014.06.14.rem2}
The conditions defined above can be shown as follows:
\begin{enumerate}
\item The TT-condition:
$$
\frac{\Gamma,T\rtr\spc \Gamma,\Delta,T'\rtr\spc \Gamma,\Delta\vdash \JJ}
{\frac{\Gamma,T\rtr\spc \Gamma,\Delta,T'\vdash\JJ}{\Gamma, T,\Delta, T'\vdash \JJ}\spc
\frac{\Gamma,T,\Delta,T'\rtr\spc \Gamma,T,\Delta\vdash \JJ}{\Gamma, T,\Delta, T'\vdash \JJ}}
$$
\item The SS-condition:
$$
\frac{\Gamma\vdash s:T\spc\Gamma,T,\Delta\vdash s':T'\spc \Gamma,T,\Delta,T'\vdash \JJ}
{\frac{\Gamma\vdash s:T\spc \Gamma,T,\Delta\vdash \JJ[s]}{\Gamma,\Delta[s]\vdash \JJ[s'][s]}
\spc 
\frac{\Gamma,\Delta[s]\vdash s'[s]:T'[s]\spc \Gamma,\Delta[s],T'[s]\vdash \JJ[s]}{\Gamma,\Delta[s]\vdash \JJ[s][s']}}
$$
\item The TS-condition:
$$
\frac{\Gamma,T\rtr \spc\Gamma,\Delta\vdash s':T'\spc \Gamma,\Delta,T'\vdash \JJ}
{\frac{\Gamma, T\rtr \spc \Gamma,\Delta\vdash \JJ[s']}{\Gamma,T,\Delta\vdash \JJ[s']}
\spc 
\frac{\Gamma,T,\Delta \vdash s':T'\spc \Gamma,T,\Delta,T'\vdash \JJ}{\Gamma,T,\Delta\vdash \JJ[s']}}
$$
\item The ST-condition:
$$
\frac{\Gamma\vdash s:T\spc\Gamma,T,\Delta,T'\rtr \spc \Gamma,T,\Delta\vdash \JJ}
{\frac{\Gamma\vdash s:T\spc \Gamma,T,\Delta,T'\vdash \JJ[s]}{\Gamma,\Delta[s], T'[s]\vdash \JJ[s]}
\spc 
\frac{\Gamma,\Delta[s],T'[s]\rtr\spc \Gamma,\Delta[s]\vdash \JJ[s]}{\Gamma,\Delta[s],T'[s]\vdash \JJ[s]}}
$$
\item The STid-condition:
$$
\frac{\Gamma\vdash s:T\spc\Gamma, T\rtr\spc\Gamma\vdash\JJ}{\frac{\Gamma\vdash s:T\spc \Gamma,T\vdash \JJ}{\Gamma\vdash\JJ[s]}}
$$
\item The $\delta$T-condition:
$$
\frac{\Gamma,T\rtr \spc \Gamma,\Delta,x:T'\rtr}{
\frac{\Gamma,T\rtr\spc \Gamma,\Delta,x:T'\vdash x:T'}{\Gamma,T,\Delta,x:T'\vdash x:T'}
\frac{\Gamma,T,\Delta,x:T'\rtr}{\Gamma,T,\Delta,x:T'\vdash x:T'}}
$$
\item The $\delta$S-condition:
$$
\frac{\Gamma\vdash s:T \spc \Gamma,T,\Delta,x:T'\rtr}{
\frac{\Gamma\vdash s:T\spc \Gamma,T, \Delta,x:T'\vdash x:T'}{\Gamma,\Delta[s],x:T'[s]\vdash x:T'[s]}
\frac{\Gamma,\Delta[s],x:T[s]'\rtr}{\Gamma,\Delta[s],x:T'[s]\vdash x:T'[s]}}
$$
\item The $\delta$Sid-condition:
$$
\frac{\Gamma\vdash s:T \spc \Gamma, x:T\rtr}{
\frac{\Gamma\vdash s:T\spc \Gamma,x:T\vdash x:T}{\Gamma\vdash s:T}}$$
\item The S$\delta$T-condition:
$$
\frac{\Gamma, y:Y,\Delta\vdash{\cal J}}{\frac{\Gamma,y_1:Y,y:Y,\Delta\vdash{\cal J}\spc \Gamma,y_1:Y\vdash y_1:Y}{\Gamma,y_1:Y,\Delta[y_1/y]\vdash{\cal J}[y_1/y]}}$$
\end{enumerate}
\end{remark}
\begin{lemma}
\llabel{2014.10.20.l1}
\llabel{2014.10.16.l1}
Let B be a unital B0-system and let $\delta_1$, $\delta_2$ be two families of operations as in Definition \ref{2014.10.20.def1}. Suppose that both $\delta_1$ and $\delta_2$ satisfy the $\delta T$, $\delta S id$ and $S\delta T$ conditions. Then $\delta_1=\delta_2$.
\end{lemma}
\begin{proof}
We have:
$$\delta_1(GT)=\wt{S}(\delta_2(GT),\wt{T}(GT,\delta_1(GT)))=\wt{S}(\delta_2(GT),\delta_1(T(GT,GT)))=\delta_2(GT)$$
where the first equality is the $S\delta T$-condition for $\delta_2$, the second equality is the $\delta T$-condition for $\delta_1$ and the third equality is the $\delta S id$-condition for $\delta_1$. 
\end{proof}
\begin{definition}
\llabel{2014.10.10.def2a}
\llabel{2014.10.20.def4}
A non-unital B-system is a non-unital B0-system that satisfy the conditions $TT$, $SS$, $TS$, $ST$ and $STid$ of Definition \ref{2014.10.16.def2}. 
\end{definition}
\begin{definition}
\llabel{2014.10.10.def2b}
\llabel{2014.10.20.def5}
A unital B-system is a unital B0-system that satisfy the conditions $TT$, $SS$, $TS$, $ST$, $STid$ of Definition \ref{2014.10.16.def2} and the conditions $\delta T$, $\delta S$, $\delta S id$ and S$\delta$T of Definition \ref{2014.10.20.def3}. 

Equivalently, a unital B-system is non-unital B-system such that there exists a family of operations $\delta$ satisfying the conditions $\delta T$, $\delta S$, $\delta S id$ and S$\delta$T of Definition \ref{2014.10.20.def3}. 
\end{definition}
\begin{example}
\rm\llabel{2014.10.20.ex1}
While being unital is a property of non-unital B-systems not 
any homomorphism of non-unital B-systems preserves units. Here is a sketch of an example 
of a homomorphism that does not preserve units.

Consider the following pairs of a monad and a left module over it. In both cases $pt$ is the constant functor corresponding to the one point set $\{T\}$ that has a unique left module structure over any monad.
\begin{enumerate}
\item $(R_1,pt)$  where $R_1$ is the monad corresponding to one unary operation $s_1(x)$ and the relation 
$$s_1(s_1(x))=s_1(x)$$
\item $(R_2,pt)$ where $R_2$ is the monad corresponding to two unary operations $s_1(x)$ and $s_2(x)$ and relations:
$$s_1(s_1(x))=s_1(x)\spc s_1(s_2(x))=s_1(x)\spc s_2(s_1(x))=s_1(x)\spc s_2(s_2(x))=s_2(x)$$
\end{enumerate}
Consider the unital B-systems $uB(R_1,pt)$ and $uB(R_2,pt)$. In $uB(R_1,pt)$ consider the non-unital sub-B-system $nuB_1$ generated by $(T\vdash s_1(1):T)$. In $uB(R_2,pt)$ consider the non-unital sub-B-system $nuB_2$ generated by $(T\vdash s_1(1):T)$ and $(T\vdash s_2(1):T)$. 

Observe that both $nuB_1$ and $nuB_2$ are in fact unital with the unit in the first one given by $(T,\dots,T\vdash s_1(n):T)$ and unit in the second one is given by $(T,\dots,T\vdash s_2(n):T)$ 
where $n$ is the number of $T$'s before the turnstile $\vdash$ symbol. 

We also have an obvious (unital) homomorphism from $uB(R_1,pt)$ to $uB(R_2,pt)$ that defines a homomorphism $nuB_1\sr nuB_2$ and that latter homomorphism is not unital. 
\end{example}

\begin{remark}
\rm
For a unital B-systems operations $S$ and $T$ can be expressed as follows.
\begin{eq}\llabel{2014.10.14.eq1}
T(Y,X)=\left\{
\begin{array}{ll}
Y&\mbox{\rm if $l(X)=l(Y)-1$}\\
ft(\partial(\wt{T}(Y,\delta(X))))&\mbox{\rm if $l(X)\ge l(Y)$}
\end{array}
\right.
\end{eq}
\begin{eq}\llabel{2014.10.14.eq2}
S(s,X)=\left\{
\begin{array}{ll}
ft(\partial(s))&\mbox{\rm if $l(X)=l(\partial(s))$}\\
ft(\partial(\wt{S}(s,\delta(X))))&\mbox{\rm if $l(X)> l(\partial(s))$}
\end{array}
\right.
\end{eq}
\end{remark}

I would like to end this section with the formulation of the following problem. I am reasonably sure that it has a straightforward solution.
\begin{problem}
\llabel{2014.10.10.prob2}
To show that a unital B0-system is isomorphic to a unital B0-system of the form $uB(CC)$ if and only if it is a unital B-system.
\end{problem}

\subsection{B-systems in Coq}

While our main interest is in pre-B-systems and B-systems in sets we would like to be able to formalize their definitions in Coq  without assuming that $B_n$ and $\wt{B}_{n+1}$ are of h-level 2. 

This suggests the following reformulation of our definitions. In what follows we give a presentation of non-unital B-systems in ``functional terms''. The presentation of the axioms related to the $\delta$-operations is more complex as can be see already in the case of the $\delta T$-axiom and we leave it for the future. 

Let us define a tower as a sequence of functions $T := (\dots\sr T_{i+1}\stackrel{p_i}{\sr}T_i\sr\dots\sr T_0)$. 

For a tower $T$ and $i,j\ge 0$ define $ft_{i}^{j}:T_{i+j}\sr T_i$ as the composition of the functions $p_k$ for $k=i,\dots,i+j-1$. When no ambiguity can arise we will write $ft^j$ instead of $ft_{i}^{j}$ and we will write $ft$ instead of $ft^1$. 

For a tower $T$, $i\ge 0$ and $G\in T_i$ define a new tower $T(G)$ setting:
$$T(G)_j=\{GD\in T_{i+j}| ft_{i}^{j}(x)=G\}$$
and defining the functions $T(G)_{j+1}\sr T(G)_j$ in the obvious way. More categorically this can expressed by saying that $T(G)_j$ is defined by the standard (homotopy) pull-back square 
$$
\begin{CD}
T(G)_j @>>> T_{i+j}\\
@VVV @VVft_{i}^{j}V\\
pt @>G>> T_i
\end{CD}
$$

For $G\in T_{i+j}$ we let $\phi_j(G)\in T(ft^{i+j}(G))_j$ denote the  obvious element. 

For towers $T$ and $T'$ define a function or morphism of towers $F:T\sr T'$ as a sequence of morphisms $F_i:T_i\sr T'_i$ which commute in the obvious sense with the functions $p_i$ and $p'_i$.

The identity function of towers $id_T$ and the composition of functions of towers are defined in the obvious way. 

For $T$, $i,j,k\ge 0$, $G\in T_i$ and  $GD\in T_j(G)$ we have the digrams:
$$
\begin{CD}
T(G)(GD)_k @>>> T(G)_{j+k} @>>> T_{i+(j+k)}\\
@VVV @VVft_{T(G),j}^{k}V @VVV\\
pt @>GD>> T(G)_j @>u_{G,j}>> T_{i+j}\\
@. @VVV @VVft_{T,i}^{j}V\\
{} @. pt @>G>> T_i \\
\end{CD}
\,\,\,\,\,\,\,\,\,\,\,\,\,\,\,\,
\begin{CD}
T(u_{G,i}(GD))_k @>>> T_{(i+j)+k}\\
@VVV @VVV\\
pt @>>u_{G,j}(GD)> T_{i+j}
\end{CD}
$$

which shows that we have natural equivalences (isomorphisms) 
\begin{eq}
\llabel{2014.06.12}
T(G)(GD)_k\cong T(u_{G,j}(GD))_k
\end{eq}
The equivalences (\ref{2014.06.12}) commute with the functions $p(G)(GD)_i$ and $p(u_{G,j}(GD))$ in the obvious sense and define an equivalence of towers
\begin{eq}
\llabel{2014.06.14.eq2}
T(G)(GD) \cong T(u_{G,j}(GD))
\end{eq}
\begin{remark}\rm
In the case when standard pull-backs are pull-backs in a category, the functions $u_{G,j}$ from $T_j(G)$ to $T_{i+j}$ are pull-backs of (split) monomorphisms and therefore are monomorphisms. In this case $T_k(G)(GD)$ is a sub-object of $T_{i+(j+k)}$ and  $T_k(u_{G,j}(GD))$ is a sub-object of $T_{(i+j)+k}$ which are canonically equal. Then we can say that
\begin{eq}
\llabel{2014.06.14.eq1}
T(G)(GD)_k = T(u_{G,j}(GD))_k
\end{eq}
where the equality is the equality of sub-objects of $T_{(i+j)+k}$. 

More generally, if $T_i$ are objects of h-level 2, the functions $u_{G,j}$ are of h-level 1 (monic inclusions) and we again can say that the equality (\ref{2014.06.14.eq1}) holds as the unique equality of monic sub-objects of $T_{(i+j)+k}$.
\end{remark} 

For a function $F:T\sr T'$ and $G\in T_i$ we obtain a function $F(G):T(G)\sr T'(G)$ using functoriality of standard pull-backs. 

Define a B-system carrier or a B-carrier as a pair $\B = (B, \wt{B})$ where $B$ is a tower and $\wt{B}$ is a family $\wt{B}_{i+1}$, $i\ge 0$ together with functions $\partial_i:\wt{B}_{i+1}\sr B_i$.
The B-system carriers in sets are the same as the ``type-and-term structures'' of \cite{Garner}.

We will denote the standard fiber of $\partial_i$ over $GT\in B_{i+1}$ by $\wt{B}_{GT}$. 

For a B-carrier  $\B$, $i\ge 0$ and $G\in B_i$, define a B-carrier $\B(G)$ as the pair $(B(G), \wt{B(G)})$ where 
$$\wt{B(G)}_{j+1}=\{ s \in \wt{B}_{i+j+1} | \partial(s)\in B(G)_{j+1}\}$$
or, categorically, $\wt{B(G)}_{j+1}$ is defined by the standard pull-back square
$$
\begin{CD}
\wt{B(G)}_{j+1} @>\wt{u}_{G,j+1}>> \wt{B}_{i+(j+1)}\\
@V\partial(G)VV @VV\partial V\\
B(G)_{j+1} @>u_{G,j+1}>> B_{i+(j+1)}
\end{CD}
$$

For a B-carrier $\B$, $i,j\ge 0$, $G\in B_i$ and $GD\in B_{i+j}$ the equivalence (\ref{2014.06.14.eq2}) clearly extends to an equivalence
\begin{eq}
\llabel{2014.06.14.eq3}
\B(G)(GD)\cong \B(u_G(GD))
\end{eq}

For B-carriers $\B$ and $\B'$ define a function of B-carriers $\FF: \B\sr \B'$ as a pair $\FF = ( F , \wt{F} )$ where $F:B\sr B'$ is a function of towers and for every $i\ge 0$, $\wt{F}_{i+1}$ is a function $\wt{B}_{i+1}\sr \wt{B}'_{i+1}$ which commutes in the obvious sense with the functions $\partial'$, $F_{i+1}$ and $\partial$. 

The identity function of B-carriers  $id_{\B}$ and the composition of functions of B-carriers are defined in the obvious way.

For a function of B-carriers $\FF:\B\sr \B'$ and $G\in B_i$ we obtain a function of B-carriers $\FF(G):\B(G)\sr \B'(F(G))$ using functoriality of standard pull-backs. 
\begin{definition}
\llabel{Bdata}
Non-unital B-system data is given by the following:
\begin{enumerate}
\item a B-system carrier $\B$ ,
\item an isomorphism $pt\sr B_0$,
\item for every $m\ge 0$, $Y\in B_{n+1}$ a B-carrier function $\TT_{Y}:\B(p_n(Y))\sr \B(Y)$,
\item for every $m\ge 0$, $s\in \wt{B}_{n+1}$, a B-carrier function $\SS_s:\B(\partial(s))\sr \B(p_n(\partial(s)))$,
\end{enumerate}
\end{definition}
\begin{problem}
\llabel{2014.10.10.prob1}
Construct an equivalence between the type of  non-unital B0-systems the type of non-unital B-system data such that the types $B_*$ and $\wt{B}$ are sets.
\end{problem}
\begin{construction}\rm
\llabel{2014.10.10.constr1}
A non-unital B-system carrier is the same as two families of sets $B_n$, $\wt{B}_{n+1}$ together with maps $p_n:B_{n+1}\sr B_n$ and $\partial:\wt{B_{n+1}}\sr B_{n+1}$. 

An isomorphism $pt\sr B_0$ is the same as an element $pt\in B_0$ such that for all $X\in B_0$, $X=pt$.

For a given $Y\in B_{n+1}$ a B-carrier function $\TT_{Y}:\B(ft(Y))\sr \B(Y)$ is the same as: 
\begin{enumerate}
\item for all $i\ge 0$, $X\in B_{n+i}$ such that $ft^i(X)=ft(Y)$, an element $T(Y,X)\in B_{n+i+1}$ such that $ft^i(T(Y,X))=Y$,
\item for all $i\ge 0$, $r\in \wt{B}_{n+i+1}$ such that $ft^{i+1}(\partial(r))=ft(Y)$, an element $\wt{T}(Y,r)$ such that $ft^{i+1}(\partial(r))=ft(Y)$.
\end{enumerate}
For $i=0$, the operation $T$ is uniquely determined by the condition $ft^i(T(Y,X))=Y$ which leaves us with the operations $T$ and $\wt{T}$ as in Definition \ref{2014.10.10.def1} satisfying the conditions of Lemma \ref{l2014.10.10.l1}.

The same reasoning applies to $S$, $\wt{S}$.
\end{construction}
From this point on everything is assumed to be non-unital. Let $\BD = (\B, \TT, \SS,\delta)$ be B-data and $G\in B_i$. Define B-data $\BD(G)$ over $G$ as follows. The B-carrier of $\BD(G)$ is $\B(G)$. 

For $GDT\in B(G)_{i+1}$ we need to define a B-carrier function 
$$\TT(G)_{GDT}:\B(G)(p_i(GDT))\sr \B(G)(GDT)$$
We define it through the condition of commutativity of the pentagon:
\begin{eq}
\llabel{2014.06.14.eq4}
\begin{CD}
\B(G)(p_i(GDT))@>\TT(G)_{GDT}>> \B(G)(GDT)\\
@V\cong VV @VV\cong V\\
\B(u_G(p_i(GDT)))\cong \B(p_i(u_G(GDT))) @>\TT_{\_}>> \B(u_G(GDT))
\end{CD}
\end{eq}
where the vertical equivalences are from (\ref{2014.06.14.eq3}).

Similarly for $s\in \wt{B(G)}_{j+1}$ we define a B-carrier function 
$$\SS(G)_s:\B(G)(\partial(s))\sr \B(G)(p_j(\partial(s)))$$
by the diagram:
\begin{eq}
\llabel{2014.06.14.eq5}
\begin{CD}
\B(G)(\partial(s)) @>\SS(G)_s>> \B(G)(p_j(\partial(s)))\\
@VVV @VVV\\
\B(u_G(\partial(s))) \cong \B(\partial(\wt{u}_G(s))) @>\SS(\wt{u}_G(s))>> \B(p_j(\partial(\wt{u}_G(s)))) \cong \B(u_G(p_j(\partial(s))))
\end{CD}
\end{eq}

\comment{Define for $i\ge 0$ and $GDT\in B(G)_{i+1}$ an element $\delta(G)_{GDT}$ in $\wt{B(G)}_{T(G)_{GDT}(\phi_1(GDT))}$.  

We have 
$$\wt{B(G)}_{T(G)_{GDT}(\phi(GDT))}\cong \wt{B}_{u_G(T(G)_{GDT}(\phi(GDT)))}\cong \wt{B}_{T_{u_G(GDT)}(\phi(u_G(GDT)))}$$
and we define $\delta(G)_{GDT}$ as the image of $\delta_{u_G(GDT)}$ under the inverse to this equivalence. }

We can now give formulations for the conditions TT, SS, TS, ST and STid.
\begin{definition}\llabel{2014.10.16.def3.fromold}
Let us define the following conditions on a B-system data $(\B, \TT, \SS, \delta)$:
\begin{enumerate}
\item  The TT-condition.  For any $GT\in B_{i+1}$, $GDT'\in B_{j+1}(p_i(GT))$ the pentagon of B-carrier functions 
\begin{eq}
\label{TTax}
\begin{CD}
\B(p_i(GT))(p_j(GDT')) @>\TT(p_i(GT))_{GDT'}>> \B(p_i(GT))(GDT')\\
 @V \TT_{GT}(p_j(GDT')) VV  @.\\   
 \B(GT)(T_{GT}(p_j(GDT'))) @. @V V\TT_{GT}(GDT') V\\
 @V\cong VV @.\\
 \B(GT)(p_j(T_{GT}(GDT'))) @>  \TT_{T_{GT}(GDT')}(GT) >> \B(GT)(\TT_{GT}(GDT'))
 \end{CD}
\end{eq}
commutes. 
\item The SS-condition. For any $s\in \wt{B}_{i+1}$, $s'\in \wt{B}_{j+1}(\partial(s))$ the diagram of B-carrier functions 
\begin{eq}
\label{2014.06.16.eq1}
\label{SSax}
\begin{CD}
\B(\partial(s))(\partial(s')) @>\SS(\partial(s))_{s'}>> \B(\partial(s))(p_j(\partial(s')))\\
@V\SS_s(\partial(s'))VV @VV\SS_s(p_j(\partial(s'))) V\\
\B(p_i(\partial(s)))(S_s(\partial(s')))  @. \B(p_i(\partial(s))(S_s(p_j(\partial(s'))))\\
@V\cong VV @VV \cong V\\
\B(p_i(\partial(s)))(\partial(\wt{S}_s(s'))) @>\SS(p_i(\partial(s)))_{\wt{S}_{s}(s')}>> \B(p_i(\partial(s)))(p_j(\partial(\wt{S}_s(s'))))
\end{CD}
\end{eq}
commutes. 
\item The TS-condition. For any $GT\in B_{i+1}$, $s'\in \wt{B}_{j+1}(p_i(GT))$ the diagram of B-carrier functions 
\begin{eq}
\label{2014.06.16.eq3}
\label{TSax}
\begin{CD}
\B(p_i(GT))(\partial(s')) @>\SS(p_i(GT))_{s'}>> \B(p_i(GT))(p_j(\partial(s')))\\
@V\TT_{GT}(\partial(s')) VV @VV\TT_{GT}(p_j(\partial(s'))) V\\
\B(GT)(T_{GT}(\partial(s'))) @. \B(GT)(T_{GT}(p_j(\partial(s'))))\\
@V\cong VV @VV\cong V\\
\B(GT)(\partial(\wt{T}_{TG}(s'))) @>\SS(GT)_{\wt{T}_{GT}(s')}>> \B(GT)(p_j(\partial(\wt{T}_{GT}(s'))))
\end{CD}
\end{eq}
\item The ST-condition. For any $s\in \wt{B}_{i+1}$, $GTDT'\in B_{j+1}(\partial(s))$ the diagram of B-carrier functions 
\begin{eq}
\label{2014.06.16.eq2}
\label{STax}
\begin{CD}
\B(\partial(s))(p_j(GTDT')) @>\TT(\partial(s))_{GTDT'}>> \B(\partial(s))(GTDT') \\
@V\SS_s(p_j(GTDT')) VV @.\\
\B(p_i(\partial(s)))(S_s(p_j(GTDT'))) @. @VV\SS_s(GTDT') V\\
@V\cong VV @.\\
\B(p_i(\partial(s)))(p_j(S_s(GTDT'))) @>\TT(p_i(\partial(s)))_{S_s(GTDT')}>> \B(p_i(\partial(s)))(S_s(GTDT'))
\end{CD}
\end{eq}
\item The STid-condition. For any $s\in \wt{B}_{i+1}$ 
one has 
$$( \B(p_i(\partial(s)))\stackrel{T_{\partial(s)}}{\sr}\B(\partial(s))\stackrel{S_s}{\sr}\B(p_i(\partial(s))) ) = id_{\B(p_i(\partial(s))}$$
\end{enumerate}
\end{definition}
Formulation of the remaining four conditions that involve $\delta$ is more difficult since their formulation using this approach leads to conditions that  depend on the conditions from the first group. We leave their study for the future. 

\def\cprime{$'$}

\end{document}